\documentclass[11pt]{amsart}
\usepackage{latexsym}
\usepackage{fullpage}
\usepackage{amssymb}
\usepackage{amscd}
\usepackage{float}
\usepackage{graphicx}
\usepackage{amsfonts}
\usepackage{pb-diagram}
\usepackage{amsmath,amscd}
\usepackage{caption}
\usepackage{subcaption}
\usepackage{xcolor}

\newtheorem{thm}{Theorem}

\newtheorem{prop}{Proposition}
\newtheorem{defn}{Definition}
\newtheorem{remark}{Remark}

\makeatletter
\@namedef{subjclassname@2020}{\textup{2020} Mathematics Subject Classification}
\makeatother

\begin{document}

\title[Knotoids, pseudo knotoids, Braidoids and pseudo braidoids on the Torus]
  {Knotoids, pseudo knotoids, Braidoids and pseudo braidoids on the Torus}

\author{Ioannis Diamantis}
\address{China Agricultural University,
International College Beijing, No.17 Qinghua East Road, Haidian District,
Beijing, {100083}, P. R. China.}
\email{ioannis.diamantis@hotmail.com}



\setcounter{section}{-1}

\date{}

\begin{abstract}
In this paper we study the theory of knotoids and braidoids and the theory of pseudo knotoids and pseudo braidoids on the torus T. In particular, we introduce the notion of {\it mixed knotoids} in $S^2$, that generalize the notion of mixed links in $S^3$, and we present an isotopy theorem for mixed knotoids. We then generalize the Kauffman bracket polynomial, $<;>$, for mixed knotoids and we present a state sum formula for $<;>$. We also introduce the notion of {\it mixed pseudo knotoids}, that is, multi-knotoids on two components with some missing crossing information. More precisely, we present an isotopy theorem for mixed pseudo knotoids and we extend the Kauffman bracket polynomial for pseudo mixed knotoids. Finally, we introduce the theories of {\it mixed braidoids} and {\it mixed pseudo braidoids} as counterpart theories of mixed knotoids and mixed pseudo knotoids respectively. With the use of the $L$-moves, that we also introduce here for mixed braidoid equivalence, we formulate and prove the analogue of the Alexander and the Markov theorems for mixed knotoids. We also formulate and prove the analogue of the Alexander theorem for mixed pseudo knotoids.
\bigbreak
\noindent 2020 {\it Mathematics Subject Classiffication.} 57K10, 57K12, 57K14, 57K35, 57K45, 57K99, 20F36, 20F38, 20C08.
\bigbreak
\end{abstract}

\maketitle

\section{Introduction}\label{intro}

Knotoids were introduced by Turaev in \cite{T} as a generalization of 1-1 tangles by allowing the endpoints to be in different regions of the diagram. Equivalently, knotoids may be considered as open knotted curves in oriented surfaces $\Sigma$, generalizing the theory of classical knots. We call a knotoid in $\Sigma$, the equivalence class of knotoid diagrams in $\Sigma$ up to the equivalence relation induced by the standard Reidemeister moves that take place away from the endpoints of the knotoid diagram. Similar to the notion of classical links, we may extend knotoids to {\it multi-knotoids}, that is, a union of a knotoid diagram and a finite number of knot diagrams. In this paper we study the theory of knotoids on the torus T, by introducing the theory of {\it mixed knotoids}, that is, multi-knotoids that consist of two parts: the fixed part that forms the unknot and represents the complementary solid torus in $S^3$, and the moving part of the multi-knotoid, that represents the knotoid on T. Mixed knotoids may also be considered as {\it dichromatic multi-knotoids}, that is, multiknotoids on two components, the one forming the (fixed) unknot and the other forming a standard knotoid. We present an isotopy theorem for mixed-knotoids and we extend the Kauffman bracket polynomial for mixed-knotoids. 

\smallbreak

Pseudo knots were introduced in \cite{H} as projections on the 2-sphere with over/under information at some of the double points and pseudo knotoids were then introduced in \cite{D1} as standard knotoids with some missing crossing information, generalizing the notion of pseudo knots. In this paper we introduce and study the theory of {\it mixed pseudo knotoids} in order to study pseudo knotoid diagrams on the torus.  In particular, we present an isotopy theorem for mixed pseudo knotoids in $S^2$ and we then pass on the counterpart analogue of mixed knotoids and mixed pseudo knotoids, namely, the theories of {\it mixed braidoids} and {\it mixed pseudo braidoids} respectively. 

\smallbreak

Braidoids were introduced in \cite{GL2} where the authors present a braidoiding algorithm, as well as braidoid equivalence moves (see also \cite{GL1}). In this paper we present mixed braidoid equivalence moves and we formulate and prove the analogue of the Alexander theorem for mixed knotoids (see Theorem~\ref{alexbrst}). We also extend the notion of $L$-moves on mixed braidoids, with the use of which, we state and prove a geometric analogue of the Markov theorem for mixed braidoids (see Theorem~\ref{brisomarkst}). Similarly to pseudo braids, pseudo braidoids were introduced in \cite{D1} where the author presents pseudo braidoid equivalence moves and the analogue of the Alexander theorem for pseudo knotoids. With the use of $L$-moves on pseudo braidoids, the author also states and proves the analogue of the Markov theorem for pseudo braidoids. We conclude this paper by formulating and proving the analogue of the Alexander theorem for mixed pseudo knotoids.

\smallbreak

Pseudo knots comprise a relatively new and important model for DNA knots, since there exist cases of DNA knots that, after studying them by electron microscopes, it is hard to say a positive from a negative crossing. On the other hand, the theory of knotoids has become an important tool in the study of open proteins, and in particular, it has been used to classify entanglement in proteins, which are long chains of amino acids that sometimes form open ended knots \cite{DGBS, GGLDSK}. We believe that the results of this paper will find applications in various aspects of molecular biology.

\smallbreak

The paper is organized as follows: \S~\ref{prel} we recall all necessary results for knotoids and braidoids from \cite{T} and \cite{GL1, GL2}, and results concerning pseudo knotoids and pseudo braidoids from \cite{D1}. More precisely, we recall the definition of knotoid diagrams in $S^2$ and the analogue of the Reidemeister theorem for knotoids. We then recall the definition of standard braidoids, the braidoid equivalence moves and we also recall a braidoiding algorithm that is useful for the case of mixed knotoids. Then, using the $L$-moves, we state a geometric analogue of the Markov theorem for knotoids. Finally, we recall the definition of pseudo knotoids and pseudo braidoids from \cite{D1} and we present the analogue of the Reidemeister theorem for pseudo knotoids in $S^2$. We conclude this section with the analogues of the Alexander and Markov theorems for pseudo knotoids. In \S~\ref{knoidST} we introduce the notion of mixed knotoids in $S^2$ and we present the equivalence moves for mixed knotoids. As mentioned above, mixed knotoids generalize the notion of mixed links in $S^3$, which is a way of visualizing knots and links in arbitrary c.c.o. 3-manifolds. The theory of mixed knotoids in $S^2$ maybe considered as the theory of {\it torus-knotoids}, that is, knotoids in $S^2$ with a 1-handle attached, similar to the {\it spherical knotoids}, i.e. knotoids in $S^2$ and {\it planar knotoids}, i.e. knotoids in $\mathbb{R}^2$. In \S~\ref{kbp} we consider the torus as a {\it punctured torus} (see Figure~\ref{ptor1}) and we generalize the Kauffman bracket polynomial $<;>$ for mixed knotoids (Theorem~\ref{pkaufbst}). We also present a state sum formula for $<;>$ (Eq.~\ref{sskbst}) which is easier to be applied on a mixed knotoid diagram. In \S~\ref{mpkd1} we introduce the mixed pseudo knotoids as a generalization of pseudo links in $S^3$ and we present the analogue of the Reidemeister theorem for mixed pseudo knotoids. We also extend the Kauffman bracket polynomial for the case of of mixed pseudo knotoids. Finally, in \S~\ref{broid11} we present the theory of mixed braidoids, that is, the counterpart theory of mixed knotoids, and which generalize the notion of mixed braids of \cite{LR1}. More precisely, we start by translating the equivalence moves on mixed knotoids on the level of mixed braidoids and we introduce a (well-defined) closure operation on mixed braidoids. We show that the braidoiding algorithm of \cite{GL1} may be applied in the case of mixed braidoids, obtaining in that way the analogue of the Alexander theorem for mixed knotoids in $S^2$. We finally define $L$-moves on mixed braidoids, with the use of which, we formulate and prove a geometric analogue of the Markov theorem for mixed braidoids (Theorem~\ref{brisomarkst}). We conclude this paper by introducing the theory of mixed pseudo braidoids and by stating and proving the analogue of the Alexander theorem for mixed pseudo braidoids.

\bigbreak

\noindent \textbf{Acknowledgments}\ \ I would like to acknowledge several discussions with Dr. Nikolaos Koutsogoulas MD, which inspired and motivated me to write this paper. Moreover, financial support by China Agricultural University is gratefully acknowledged.

\section{Preliminaries}\label{prel}

In this section we recall all necessary results for knotoids, braidoids, pseudo knotoids and pseudo braidoids from \cite{T}, \cite{GL1, GL2} and \cite{D1}. More precisely, we recall the definition of knotoid diagrams in $S^2$ and the analogue of the Reidemeister theorem for knotoids. We then pass on the level of standard braidoids and we present the braidoid equivalence moves and we also recall a braidoiding algorithm that is useful for the case of mixed knotoids. Moreover, and using the $L$-moves, we present a geometric analogue of the Markov theorem for knotoids. We finally recall analogous results on pseudo knotoids and pseudo braidoids from \cite{D1}.

\subsection{Knotoids}\label{knd}

Knotoids were introduced in \cite{T} as open knotted curves in oriented surfaces. Classical knotoids, that is, knotoids in $S^2$, are open ended knot diagrams, generalizing the classical knots. More precisely:

\begin{defn}\rm
A {\it knotoid diagram} $K$ in an oriented surface $\Sigma$ is a generic immersion of the unit interval $[0, 1]$ into  
$\Sigma$ whose only singularities are transversal double points endowed with over/undercrossing data called crossings. The images of $0$ and $1$ under this immersion are called the endpoints of $K$ (leg and head of $K$ respectively) and are distinct from each other and from the double points. For an illustration is Figure~\ref{mkd}(a).
\end{defn}

\noindent Note that a knotoid diagram has a natural orientation from its leg to its head.

\begin{figure}[ht]
\begin{center}
\includegraphics[width=3.6in]{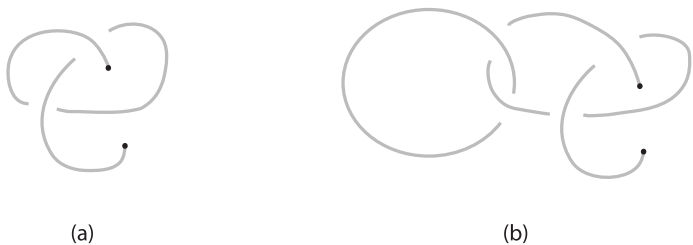}
\end{center}
\caption{(a) A knotoid and (b) a multi-knotoid.}
\label{mkd}
\end{figure}

A {\it knotoid} in $\Sigma$ is then an equivalence class of knotoid diagrams in $\Sigma$ up to the equivalence relation induced by the standard Reidemeister moves RI, RII \& RIII (see Figure~\ref{reid1}) and planar isotopy, that take place away from the endpoints. It is worth mentioning that we are not allowed to pull a strand that is adjacent to an endpoint, over or under a transversal arc, since this will result into a trivial knotoid diagram. These moves are illustrated in Figure~\ref{forb}, and they are called {\it forbidden moves} of knotoids.  

\begin{figure}[ht]
\begin{center}
\includegraphics[width=4.6in]{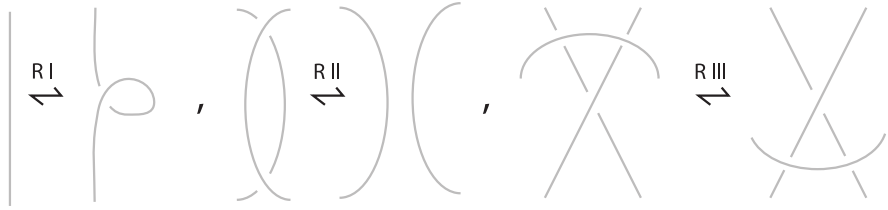}
\end{center}
\caption{The Reidemeister moves.}
\label{reid1}
\end{figure}

\begin{figure}[ht]
\begin{center}
\includegraphics[width=4.5in]{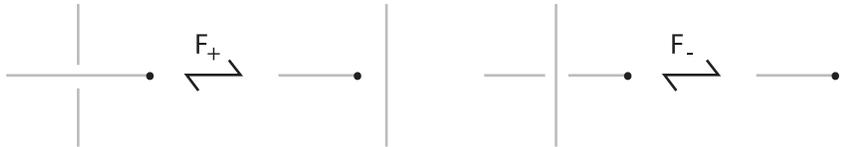}
\end{center}
\caption{The forbidden moves.}
\label{forb}
\end{figure}

Note that there are two situations where forbidden moves, seemingly occur as illustrated in Figure~\ref{ffm}. We shall call these moves
{\it fake forbidden moves}.

\begin{figure}[ht]
\begin{center}
\includegraphics[width=5.6in]{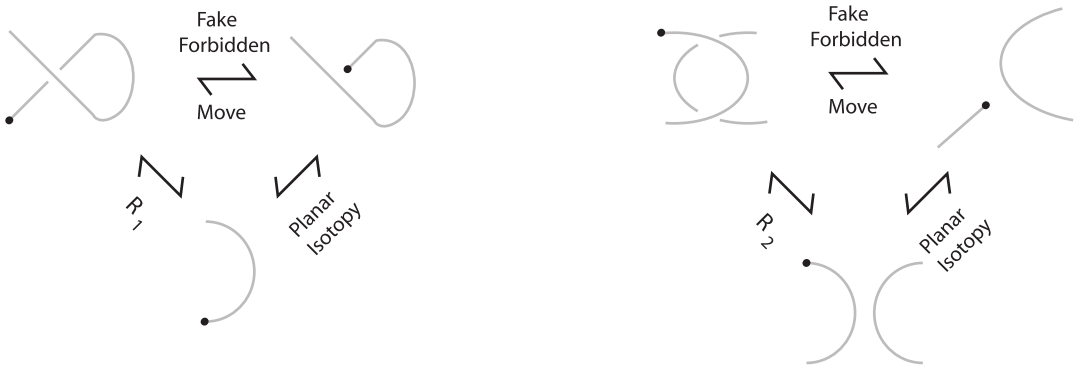}
\end{center}
\caption{Fake forbidden moves.}
\label{ffm}
\end{figure}

The definition of knotoids can be extended to linkoids and multi-knotoids as follows: a {\it linkoid diagram} is defined to be an immersion of a disjoint union of finitely many unit intervals whose images are knotoid diagrams and a {\it multi-knotoid diagram} is defined to be a union of a knotoid diagram and a finite number of knot diagrams (see Figure~\ref{mkd}(b)). Moreover, knotoid equivalence extends naturally to both linkoid diagrams and multi-knotoid diagrams, leading to the theory of linkoids and multi-knotoids. Finally, note that knotoids and linkoids are naturally oriented from the leg to the head.

\smallbreak

As noted in \cite{T}, the theory of knotoid diagrams suggests a new diagrammatic approach to knots, since every knotoid diagram determines a classical knot if we connect the endpoints of a knotoid diagram with an arc in $S^2$ that goes under or over each arc it meets. We call the resulting knots in $\mathbb{R}^3$, the {\it underpass closure} and the {\it overpass closure} of the knotoid respectively. Obviously, different closures of a knotoid may result in different knots. Thus, in order to represent knots via knotoid diagrams, we fix the closure type and we have the following result from \cite{T}:

\smallbreak

\begin{prop}
Assuming a specific closure type, there is a well-defined surjective map from knotoid diagrams to classical knots.
\end{prop}

Finally, it is worth mentioning that in \cite{D1}, the pseudo closure of knotoids is defined, that is, the arc used to connect the endpoints of a knotoid diagram in $S^2$ is missing all crossing information. This leads to the theory of {\it pseudo-knotoids} introduced and studied in \cite{D1}.

\subsection{Braidoids}\label{broid}

In \cite{GL1} braidoid diagrams are defined (similarly to classical braid diagrams), as a system of finite descending strands that involves one or two strands starting with or terminating at an endpoint that is not necessarily at top or bottom lines of the defining region of the diagram. More precisely:

\begin{defn}\label{broiddefn}\rm
A {\it braidoid diagram} $B$ is a system of a finite number of arcs immersed in $[0, 1] \times [0,1] \subset \mathbb{R}^2$, where $\mathbb{R}^2$ is identified with the xt-plane, such that the t-axis is directed downward. The arcs of $B$ are called the strands of $B$. Each strand of $B$ is naturally oriented downward, with no local maxima or minima, following the natural orientation of $[0, 1]$. Moreover, there are only finitely many intersection points among the strands, which are transversal double points endowed with over/under data, and are called crossings of $B$. 
\end{defn}

A braidoid diagram has two types of strands, the classical strands, i.e. braid strands connecting points on $[0, 1]\times \{0\}$ to points on $[0, 1] \times \{1\}$, and the {\it free strands} that either connect a point in $[0,1]\times \{0\}$ or in $[0,1]\times \{1\}$ to an {\it endpoint} located anywhere in $[0, 1]\times [0, 1]$, or they connect two endpoints that are located anywhere in $[0, 1] \times [0, 1]$. These points that don't necessarily lie on $[0, 1] \times \{0\}$ or $[0, 1] \times \{1\}$, are called {\it braidoid ends}. For an illustration see Figure~\ref{broi}. For more details and examples the reader is referred to \cite{GL1}. 

\begin{figure}[ht]
\begin{center}
\includegraphics[width=0.8in]{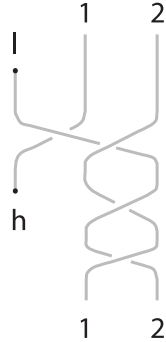}
\end{center}
\caption{A braidoid diagram.}
\label{broi}
\end{figure}

We now present braidoid isotopy:

\begin{defn}\label{broidiso}\rm
Two braidoid diagrams are said to be isotopic if one can be obtained from the other by a finite sequence of the following moves, that we call  {\it braidoid isotopy} moves:
\smallbreak
\begin{itemize}
\item[$\bullet$] {\it Braidoid} $\Delta${\it -moves} illustrated in the left part of Figure~\ref{biso1}: a $\Delta$-move replaces a segment of a strand with two segments in a triangular disk free of endpoints, passing only over or under the arcs intersecting the triangular region of the move whilst the downward orientation of the strands is preserved.
\smallbreak
\item[$\bullet$] {\it Vertical moves} as illustrated in the right part of Figure~\ref{biso1}: the endpoints of a braidoid diagram can be pulled up or down in the vertical direction but without letting an endpoint of a braidoid diagram to be pushed/pulled over or under a strand (recall the forbidden moves).
\smallbreak

\begin{figure}[ht]
\begin{center}
\includegraphics[width=3.2in]{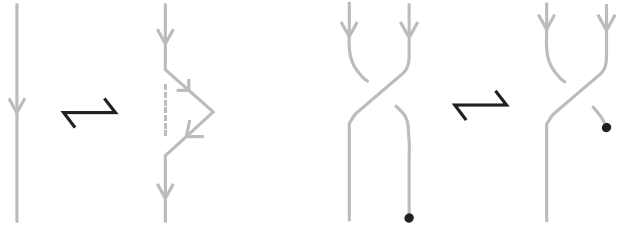}
\end{center}
\caption{A $\Delta$-move and a vertical move on a braidoid.}
\label{biso1}
\end{figure}

\item[$\bullet$] {\it Swing moves} as illustrated in Figure~\ref{biso2}: the endpoints are allowed to swing to the right or the left like a
pendulum as long as the downward orientation on the moving arc is preserved, and the forbidden moves are not violated.
\end{itemize}
\smallbreak
\begin{figure}[ht]
\begin{center}
\includegraphics[width=3.5in]{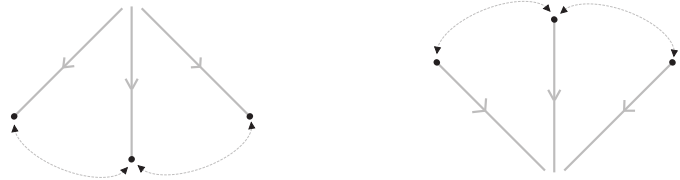}
\end{center}
\caption{The swing moves on braidoids.}
\label{biso2}
\end{figure}
\smallbreak
An isotopy class of braidoid diagrams is called a {\it braidoid}. Moreover, a {\it labeled braidoid diagram} is a braidoid diagram with a label over or under assigned to each pair of corresponding ends.
\end{defn}

We define a closure operation on labeled braidoids, which is similar to that of mixed braids in handlebodies \cite{HL}.

\begin{defn}\rm
A {\it labeled braidoid diagram} is a braidoid diagram whose corresponding ends are labeled either with ``o'' or ``u'' in pairs. The {\it closure} of a labeled braidoid is realized by joining each pair of corresponding ends by a vertical segment, either over or under the rest of the braidoid and according to the label attached to these braidoid ends (see Figure~\ref{clab}).
\end{defn}

\begin{figure}[ht]
\begin{center}
\includegraphics[width=5.5in]{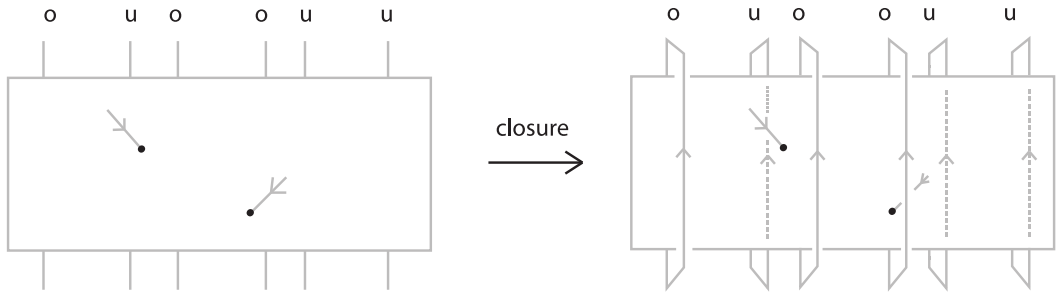}
\end{center}
\caption{The closure of a labeled braidoid.}
\label{clab}
\end{figure}

\begin{remark}\rm
It is crucial to note that different labels on the endpoints of a braidoid may yield non-equivalent closures. For more details the reader is referred to \cite{GL1, HL}. Moreover, it is worth mentioning that in \cite{GL1} the authors prove that any knotoid diagram may be isotoped to be the closure of some labeled braidoid diagram whose labels are all ``u'' (\cite{GL1} Corollary 1), and they define a {\it uniform braidoid} to be a labeled braidoid with all labels ``u''.
\end{remark}

In \cite{GL1, GL2}, the authors present a braidoiding algorithm for knotoids, with the use of which, they obtain the following result: 

\begin{thm}[{\bf The analogue of the Alexander theorem for knotoids}] \label{alexkn}
Any (multi)-knotoid diagram is isotopic to the closure of a (labeled) braidoid diagram.
\end{thm}

It is worth mentioning that the braidoiding algorithm in \cite{GL1} is identical to the braiding algorithm presented in \cite{LR1} for the classical strands of the braidoid, and similar for the free strands. We now recall the braidoiding algorithm of \cite{GL1}, which is crucial for obtaining the analogue of the Alexander theorem for knotoids on T. 

\bigbreak

The main idea of the braidoiding algorithm is to keep the arcs of the oriented mixed link diagrams that go downwards with respect to the height function unaffected, and replace arcs that go upwards with braid strands. These arcs are called {\it up-arcs} (see Figure~\ref{upa}). Note that by small perturbations in a knotoid diagram we may cancel horizontal or vertical arcs and thus, a knotoid diagram may be assumed to consist of a finite number of arcs that are oriented downwards or upwards. When/If we run along an up-arc, we subdivide it into smaller arcs, each containing crossings of one type only as shown in Figure~\ref{upa}. 

\begin{figure}[ht]
\begin{center}
\includegraphics[width=4.4in]{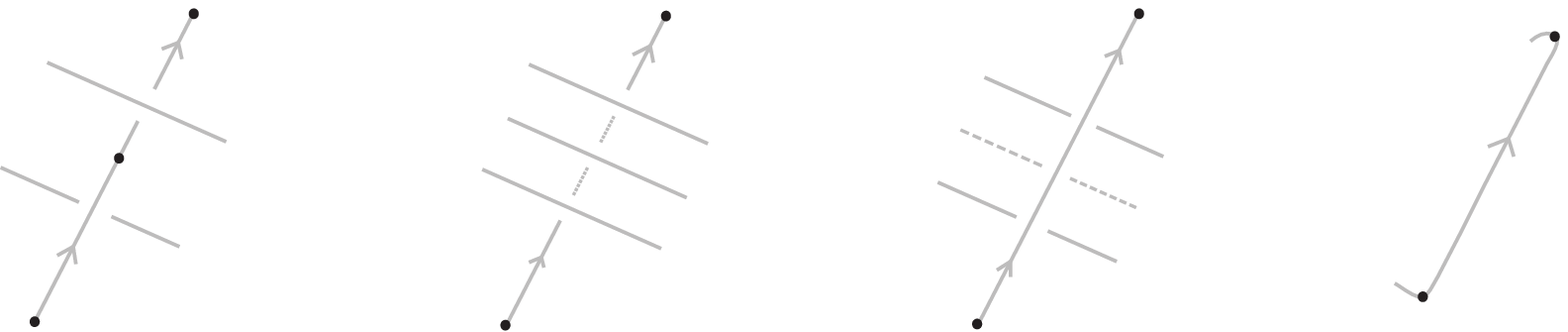}
\end{center}
\caption{Up-arcs.}
\label{upa}
\end{figure}

We now label every up-arc with an ``o''or a ``u'', according to the crossings it contains. If it contains no crossings, then the choice is arbitrary and the arc is called a {\it free up-arc}. We perform an $o$-braiding move on all up-arcs which were labeled with an ``o'' and $u$-braiding moves on all up-arcs which were labeled with a ``u'' (see Figure~\ref{ahg}).
\smallbreak

\begin{figure}[ht]
\begin{center}
\includegraphics[width=2.5in]{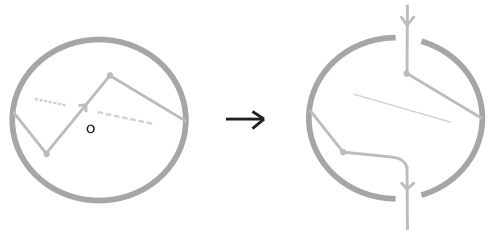}
\end{center}
\caption{Braiding moves for up-arcs.}
\label{ahg}
\end{figure}

The result is a braidoid whose closure is isotopic to the initial knotoid. For more details the reader is referred to \cite{GL1}.

\smallbreak

Moreover, in \cite{GL2}, the authors present a geometric analogue of the Markov's theorem for knotoids, with the use of the $L$-moves. More precisely, an {\it $L$-move} on a labeled braidoid $\beta$, consists in cutting an arc of $\beta$ open and pulling the upper cutpoint downward and the lower upward, so as to create a new pair of braid strands with corresponding endpoints (on the vertical line of the cutpoint), and such that both strands cross the rest of the braid all over in the case of an $L_o$-move or all under in the case of an $L_u$-move. We finally assign to the new pair of corresponding strands the label ``o'' or ``u'' according to the type of $L$-move that was applied. At this point it is crucial to observe the similarity between the $L$-moves and the ``braidoiding'' moves on the braidoiding algorithm. This similarity allow us to formulate the analogue of the Markov theorem for knotoids.

\smallbreak

In \cite{GL2} the authors also define the {\it fake forbidden moves} on a labeled braidoid diagram $B$, as forbidden moves on $B$ which upon closure induce a sequence of fake forbidden moves on the resulting (multi-)knotoid diagram. Moreover, a {\it fake swing move} is defined as a swing move which is not restricted, in the sense that the endpoint surpasses the vertical line of a pair of corresponding ends but in the closure it gives rise to a sequence of swing and fake forbidden moves on the resulting (multi-)knotoid diagram. See Figure~\ref{fm2} for an example of a fake swing move and a fake forbidden move on a labeled braidoid diagram.

\begin{figure}[ht]
\begin{center}
\includegraphics[width=3.7in]{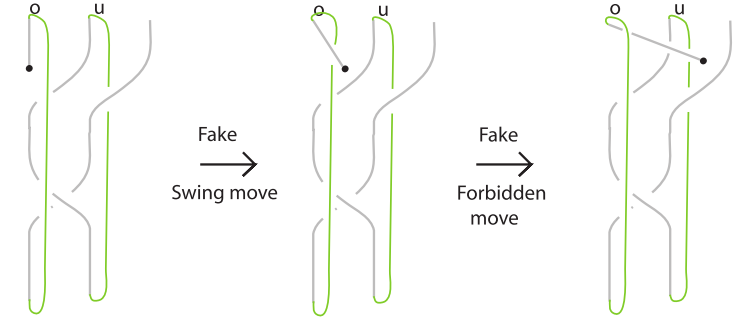}
\end{center}
\caption{A fake swing move and a fake forbidden move.}
\label{fm2}
\end{figure}

As proved in Lemma~8 \cite{GL2}, a fake forbidden move can be generated by a sequence of $L$-moves, together with planar isotopy and fake swing moves. Define now $L${\it -equivalence} on labeled braidoid diagrams to be the equivalence relation on labeled braidoid diagrams generated by the $L$-moves together with labeled braidoid isotopy moves and fake swing moves. $L$-equivalence turns out to be especially useful for formulating a braidoid equivalence since there is no algebraic structure for braidoids. Indeed, we have the following result \cite{GL2}:

\begin{thm}[{\bf An analogue of the Markov theorem for braidoids}]\label{brisomark}
The closures of two labeled braidoid diagrams are isotopic (multi)-knotoids in $\Sigma$ if and only if the labeled braidoid diagrams are related to each other via $L$-equivalence moves.
\end{thm}

\subsection{Pseudo knotoids}

In this subsection we recall basic results on pseudo knotoids, introduced and studied in \cite{D1}, as a generalization of pseudo knots (see also \cite{HJMR, D2}.)

\begin{defn}\rm
A {\it pseudo knotoid diagram} in a surface $\Sigma$ is a knotoid diagram in $\Sigma$ where some crossing information may be missing. The undetermined crossings are again called {\it pre-crossings} (for an illustration see Figure~\ref{pcl}).
\end{defn}

\begin{figure}[ht]
\begin{center}
\includegraphics[width=1.8in]{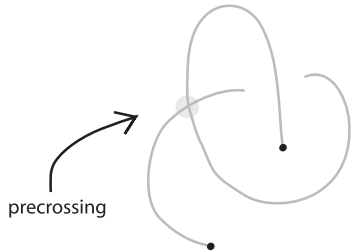}
\end{center}
\caption{A pseudo knotoid.}
\label{pcl}
\end{figure}

A pseudo knotoid in $\Sigma$ is then an equivalence class of pseudo knotoid diagrams in $\Sigma$ up to the equivalence relation induced by moves illustrated in Figure~\ref{reid} and planar isotopy.

\begin{figure}[ht]
\begin{center}
\includegraphics[width=6.2in]{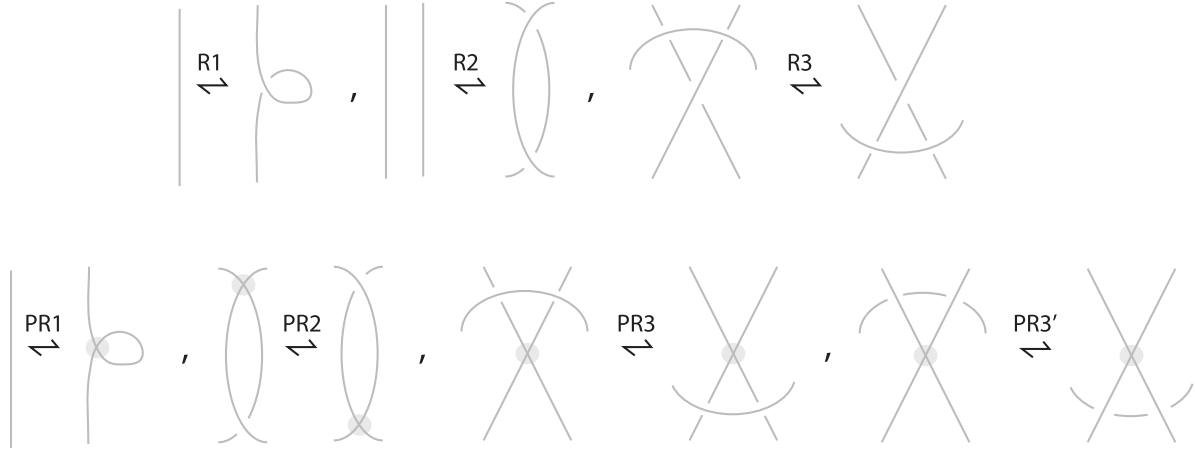}
\end{center}
\caption{Reidemeister moves for pseudo knots.}
\label{reid}
\end{figure}

\begin{remark}\rm
It is worth mentioning that the theory of pseudo knotoids is close related to the theory of singular knotoids, that is, knotoids with finite many singularities. 
\end{remark}

As illustrated in Figure~\ref{pfm}, there is an extra forbidden move in the case of pseudo knotoids, similar to the first forbidden move in the case of knotoids, tha we call {\it pseudo forbidden move}.

\begin{figure}[ht]
\begin{center}
\includegraphics[width=2in]{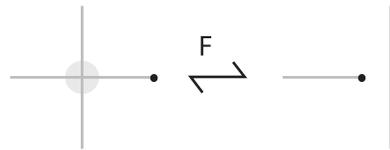}
\end{center}
\caption{The pseudo forbidden move.}
\label{pfm}
\end{figure}

Note that there is a situation where pseudo forbidden moves, seemingly occur as illustated in Figure~\ref{πffm}. We shall call these moves
{\it fake pseudo forbidden moves}.

\begin{figure}[ht]
\begin{center}
\includegraphics[width=2.2in]{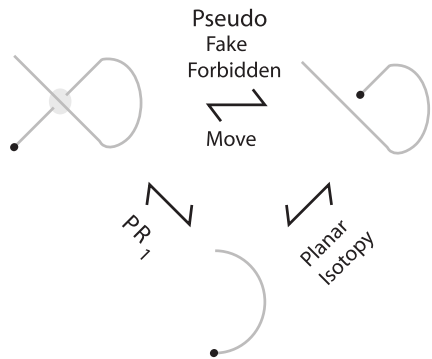}
\end{center}
\caption{The fake pseudo forbidden move.}
\label{πffm}
\end{figure}

\subsection{Pseudo braidoids}

In the same way that {\it braidoids} were defined, we may define {\it pseudo braidoids} as the counterpart theory of pseudo knotoids. Pseudo braidoids are (labeled) braidoids with some crossing information missing. In order to obtain pseudo (labeled) braidoid isotopy, we allow the analogue of the vertical move on pseudo braidoids, that we call {\it pseudo vertical  move}, illustrated in Figure~\ref{biso3}.

\begin{figure}[ht]
\begin{center}
\includegraphics[width=1.35in]{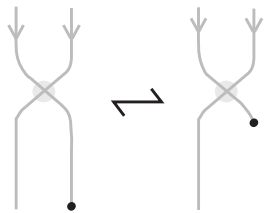}
\end{center}
\caption{The pseudo vertical move.}
\label{biso3}
\end{figure}

Note also that, as illustrated in Figure~\ref{fmb2}, a new type of forbidden move appears whenever an endpoint is to be pushed or be pulled away from a pre-crossing. We call such a move, a {\it pseudo forbidden move}.

\begin{figure}[ht]
\begin{center}
\includegraphics[width=1.8in]{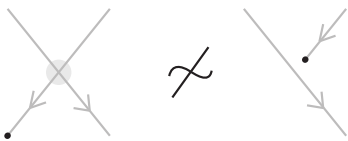}
\end{center}
\caption{The pseudo forbidden move on pseudo braidoids.}
\label{fmb2}
\end{figure}

If we define the closure operation of pseudo braidoids in the same way as in the case of classical braidoids, we obtain the following result:

\begin{thm}[{\bf The analogue of the Alexander theorem for pseudo knotoids}] \label{alexpknoid}
Every pseudo (multi)-knotoid can be obtained by closing a pseudo braidoid.
\end{thm}

We define pseudo braidoid equivalence as follows:

\begin{defn}\label{pbroidiso}\rm
Two pseudo braidoid diagrams are said to be isotopic if one can be obtained from the other by a finite sequence of the moves of Definition~\ref{broidiso} together with the pseudo vertical move. An isotopy class of pseudo braidoid diagrams is called a {\it pseudo braidoid} and a {\it labeled pseudo braidoid diagram} is a pseudo braidoid diagram with a label over or under assigned to each pair of corresponding ends.
\end{defn}

Let now {\it pseudo $L$-moves} on labeled pseudo braidoids, be $L$-moves on labeled braidoids, that is, the two strands that appear after the performance of a pseudo $L$-move will cross the rest of the braid only with real crossings and in particular, all over the rest of the braidoid in the case of an $L_o$-move or all under the rest of the braidoid in the case of an $L_u$-move. Moreover, define a {\it fake swing move} to be a swing move which is not restricted, in the sense that the endpoint surpasses the vertical line of a pair of corresponding ends, but in the closure it gives rise to a sequence of swing and fake pseudo forbidden moves on the resulting pseudo (multi-)knotoid diagram. Then, in \cite{D1}, the pseudo $L$-equivalence on pseudo braidoid diagrams is defined as follows:

\begin{defn}\rm
The pseudo $L$-moves together with labeled pseudo braidoid isotopy moves and fake swing moves, generate an equivalence relation on labeled pseudo braidoid diagrams that is called {\it pseudo} $L${\it -equivalence}.
\end{defn}

Finally, using the pseudo $L$-equivalence, the analogue of the Markov theorem for pseudo braidoids is obtained \cite{D1}. More precisely, we have the following result:

\begin{thm}[{\bf An analogue of the Markov theorem for pseudo braidoids}]\label{breqpbroid}
The closures of two labeled pseudo braidoid diagrams are isotopic pseudo (multi)-knotoids in $\Sigma$ if and only if the labeled pseudo braidoid diagrams are related to each other via pseudo $L$-equivalence moves.
\end{thm}

\section{Torus Knotoids}\label{knoidST}

In this section we introduce and study the theory of {\it mixed knotoids} in $S^2$, that generalize the theory of mixed links in $S^3$ of \cite{LR1}. We present the analogue of the Reidemeister theorem for mixed knotoids and we extend the Kauffman bracket polynomial for mixed knotoids using the skein relations and a state sum formula.

\subsection{Mixed knotoids and Isotopy}

Consider $S^3$ to be the union of two solid tori. We may consider ST to be the complement of the other Solid Torus, $\widehat{I}$, in $S^3$. Then, we may represent a knotoid $K$ on T by an oriented multi-knotoid $\widehat{I}\cup K$ in $S^{2}$, that consists of the unknotted {\it fixed part} $\widehat{I}$, that represents the complementary solid torus in $S^3$, and the {\it moving part} $K$ that links with $\widehat{I}$ and that represents the knotoid on T (see Figure~\ref{mknoid}). We will call such multi-knotoids as \textit{mixed knotoids}, similar to classical mixed links in $S^3$ that represents links in ST and in any c.c.o. 3-manifold (for more details the reader is referred to \cite{LR1, La1, DL1}). A \textit{mixed knotoid diagram} then is a diagram $\widehat{I}\cup \widetilde{K}$ of $\widehat{I}\cup K$ on the plane of $\widehat{I}$, where this plane is equipped with the top-to-bottom direction of $I$.

\begin{figure}[!ht]
\begin{center}
\includegraphics[width=2.4in]{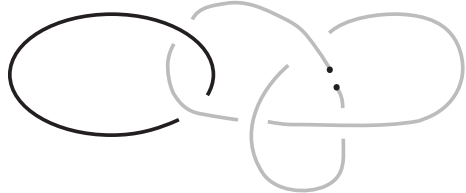}
\end{center}
\caption{A mixed knotoid.}
\label{mknoid}
\end{figure}

Let now $K_1, K_2$ be two knotoids on T. It follows that $K_1$ is isotopic to $K_2$ if and only if the mixed knotoids $\widehat{I}\cup \widetilde{K_1}$ and $\widehat{I}\cup \widetilde{K_2}$ are isotopic in $S^2$ by equivalence moves that keep $\widehat{I}$ pointwise fixed. More precisely, in terms of diagrams we have the following result for isotopy of mixed knotoids:

\begin{thm}[{\bf The analogue of the Reidemeister theorem for mixed knotoids}]\label{RthmST}
Two mixed knotoids diagrams are isotopic if and only if they differ by a finite sequence of the following moves:
\smallbreak
\begin{itemize}
\item[i.] The classical Reidemeister moves that only involve the moving part of the mixed knotoid and that take place away from the endpoints of the mixed knotoid (recall Figure~\ref{reid1}).
\smallbreak
\item[ii.] The generalized Reidemeister moves $GR_2$ and $GR_3$, that involve both the fixed and the moving part of the mixed knotoid, and which are illustrated in Figure~\ref{mreidm}.
\smallbreak
\item[iii.] The {\it generalized (or mixed) fake forbidden moves} illustrated in Figure~\ref{almo}.
\end{itemize}
\end{thm}

\begin{figure}[!ht]
\begin{center}
\includegraphics[width=3.5in]{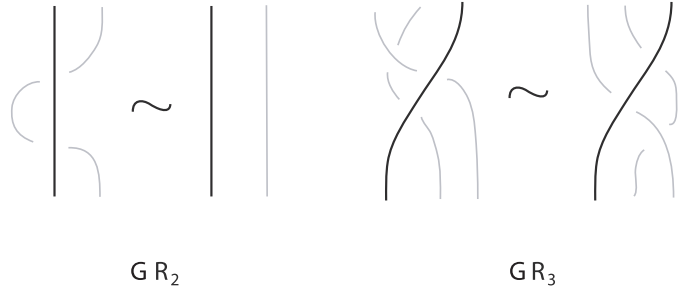}
\end{center}
\caption{The generalized Reidemeister moves.}
\label{mreidm}
\end{figure}

\begin{figure}[H]
\begin{center}
\includegraphics[width=3.4in]{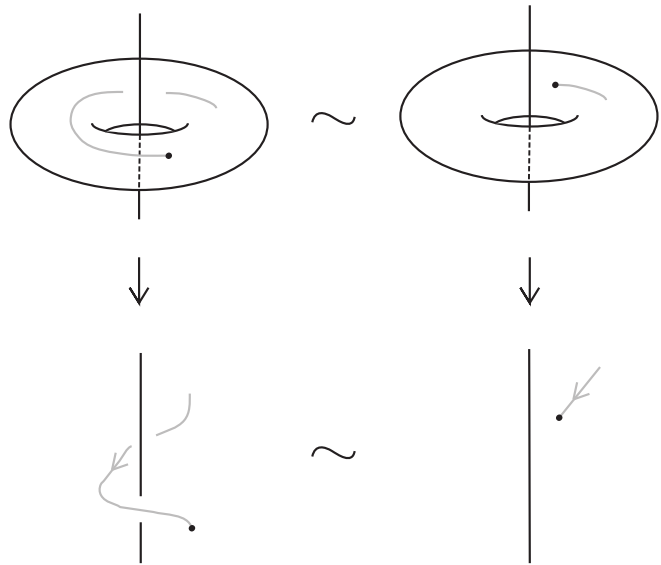}
\end{center}
\caption{The generalized fake forbidden moves.}
\label{almo}
\end{figure}

\subsection{The Kauffman bracket polynomial for mixed knotoids}\label{kbp}

In this subsection we recall the definition of the Kauffman bracket polynomial for knotoids \cite{T} and we extend the Kauffman bracket polynomial for mixed knotoids.

\bigbreak

In \cite{T} the Kauffman bracket polynomial, $<;>$, is defined for knotoids, extending the Kauffman bracket polynomial for classical knots presented in \cite{LK}. More precisely, we have the following:

\begin{defn}\label{pkaufb}\rm
Let $L$ be a knotoid. The {\it Kauffman bracket polynomial} of $L$ is defined by means of the following relations: 
\begin{figure}[H]
\begin{center}
\includegraphics[width=3in]{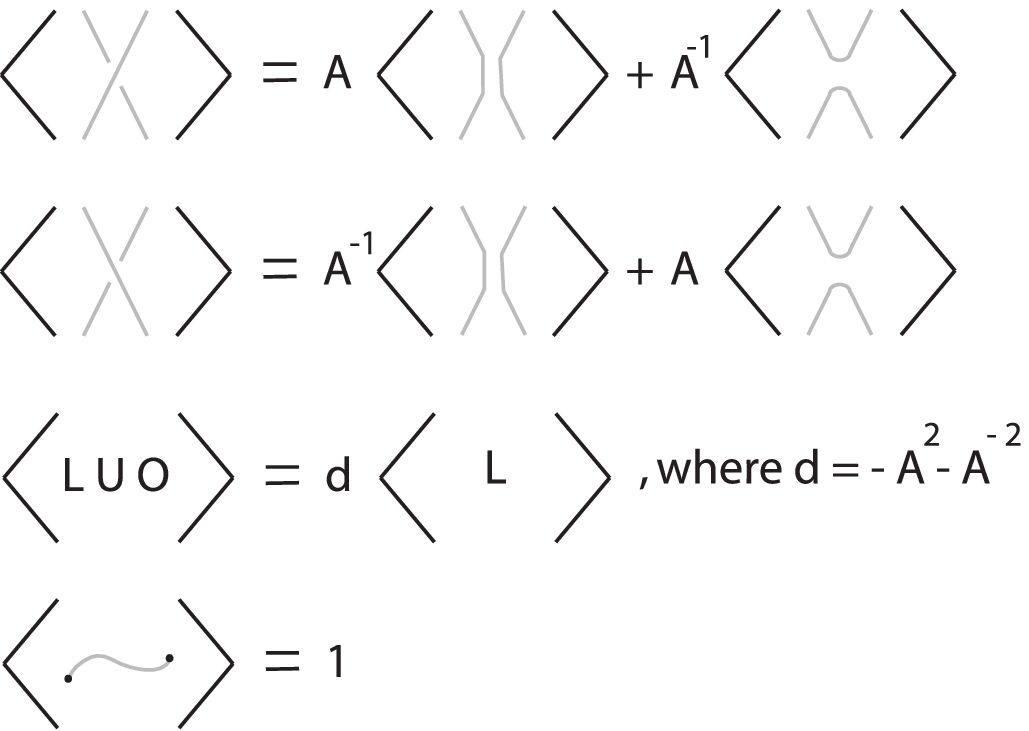}
\end{center}
\label{pkb}
\end{figure}
\end{defn}

As in the case of the Kauffman bracket polynomial for standard links in $S^3$, we may normalize the Kauffman bracket for knotoids in $S^2$ by considering the product of $<L>$ by the factor $\left( -A^3\right) ^{-wr(L)}$, where $wr(L)$ is the writhe of the knotoid $L$, defined as the number of positive crossings minus the number of negative crossings of $L$ (see Figure~\ref{si}).

\begin{figure}[H]
\begin{center}
\includegraphics[width=1.6in]{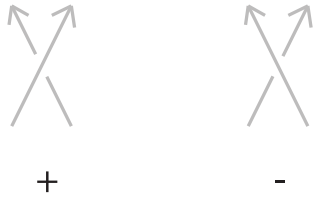}
\end{center}
\caption{The sign of the crossings.}
\label{si}
\end{figure}

As explained in \cite{T}, the normalized Kauffman bracket polynomial for knotoids in $S^2$ generalizes the Jones polynomial of classical knots in $S^3$ under the substitution $A=t^{-1/4}$.

\bigbreak

We now extend the definition of the Kauffman bracket polynomial for mixed knotoids in $S^2$. For this, it is more convenient to view T as a punctured torus (for an illustration of a knot in this set up see Figure~\ref{ptor1}. Note that in this setting, two knotoids $K_1$ and $K_2$ are isotopic if and only if their diagrams differ by a finite sequence of the standard three Reidemeister moves in the punctured plane, that take place away from the endpoints of the knotoids. Moreover, a knot that goes once around the dot and has no crossings will be called a {\it mixed-unknot} (see Figure~\ref{ptor1}). Note also that we will still refer to knotoids in this setting as mixed knotoids.

\begin{figure}[ht]
\begin{center}
\includegraphics[width=2.6in]{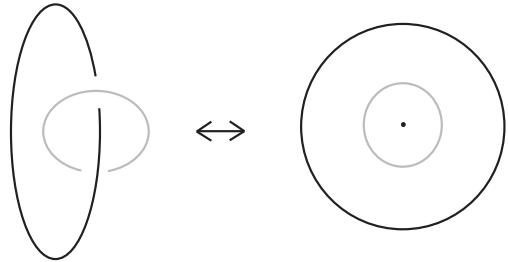}
\end{center}
\caption{T as punctured torus and a mixed-unknot.}
\label{ptor1}
\end{figure}

We are now in position to define the Kauffman bracket polynomial for mixed knotoids.

\begin{defn}\label{pkaufbst}\rm
Let $L$ be a mixed knotoid. The {\it Kauffman bracket polynomial} of $L$ is defined by means of the relations in Definition~\ref{pkaufb} together with the following relations:
\begin{figure}[ht]
\begin{center}
\includegraphics[width=1.6in]{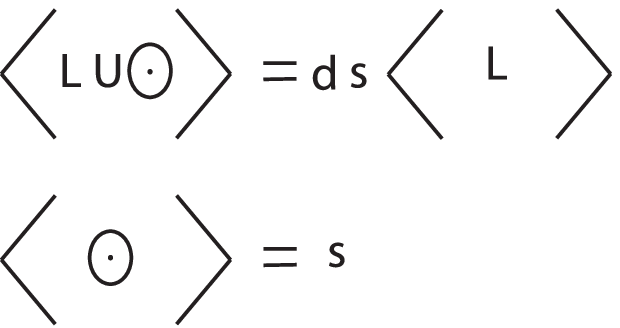}
\end{center}
\label{pkbst}
\end{figure}
\end{defn}

Simple computations show that $<L>$, which is a Laurent polynomial in $\mathbb{Z}\left[A^{\pm 1}, s\right]$, is invariant under all isotopy moves except from RI. Similarly to the case of knotoids in $S^2$ we may normalize it and obtain an invariant for mixed knotoids in $S^2$. Indeed, we have the following:

\begin{thm}
Let $L$ be a mixed knotoid diagram. The polynomial
\[
P_L(A, s)\ =\ (-A^3)^{w(L)}\, <L>,
\]
\noindent where $w(L)$ is the writhe of the knotoid and $<L>$ the Kauffman bracket polynomial of $L$, is an invariant of mixed knotoids.
\end{thm}

An equivalent definition of $<L>$ is obtained by using the {\it states} of a mixed knotoid, similar to the states of knots. More precisely, let $L$ be a mixed knotoid with $n$ crossings. We resolve each crossing by applying the A-smoothing when the crossing is positive and the B-smoothing otherwise, and we obtain a single embedded segment, several embedded unknots and several embedded mixed-unknots. Let now $S(L)$ denote the set of all possible states of $L$, $\sigma_S\in \mathbb{Z}$ denotes the sum of $\pm 1$ of $S$ over all crossings of $L$, that is, the number of A-smoothings minus the number of B-smoothings in $S$, $|S|$ denotes the number of embedded unknots and mixed-unknots and let $|M|$ denote the number of mixed-unknots only. Then, the Kauffman bracket polynomial of $L$ may be defined as:

\begin{equation}\label{sskbst}
<L>\ =\ \underset{s\in S(L)}{\sum}\, A^{\sigma_s}\, \left(-A^2-A^{-2}\right)^{|S|-1}\, t^{|M|}.
\end{equation}

The proof that $<L>$ is an invariant of mixed knotoids, that is, the value of $<L>$ is not changed whenever an isotopy move is performed, follows by simple calculations.

\subsection{A note on skein modules}

Skein modules were independently introduced by Przytycki \cite{P} and Turaev \cite{Tu} as generalizations of knot polynomials in $S^3$ to knot polynomials in arbitrary 3-manifolds. The essence is that skein modules are quotients of free modules over ambient isotopy classes of links in 3-manifolds by properly chosen local (skein) relations. In particular, let $M$ be an oriented $3$-manifold and $\mathcal{L}_{{\rm fr}}$ be the set of isotopy classes of unoriented framed links in $M$. Let $R=\mathbb{Z}[A^{\pm1}]$ be the Laurent polynomials in $A$ and let $R\mathcal{L}_{{\rm fr}}$ be the free $R$-module generated by $\mathcal{L}_{{\rm fr}}$. Let $\mathcal{S}$ be the ideal generated by the skein expressions $L-AL_{\infty}-A^{-1}L_{0}$ and $L \bigsqcup {\rm O} - (-A^2-A^{-1})L$, where $L_{\infty}$ and $L_{0}$ are represented schematically by the illustrations in Figure~\ref{skein}. Note that blackboard framing is assumed. 

\begin{figure}[!ht]
\begin{center}
\includegraphics[width=1.9in]{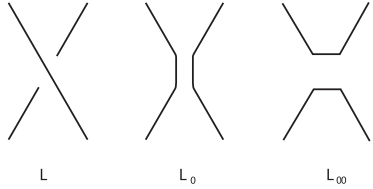}
\end{center}
\caption{The links $L$, $L_0$ and $L_{\infty}$ locally.}
\label{skein}
\end{figure}

Then the {\it Kauffman bracket skein module} of $M$, KBSM$(M)$, is defined to be:

\begin{equation*}
{\rm KBSM} \left(M\right)={\raise0.7ex\hbox{$
R\mathcal{L} $}\!\mathord{\left/ {\vphantom {R\mathcal{L_{{\rm fr}}} {\mathcal{S} }}} \right. \kern-\nulldelimiterspace}\!\lower0.7ex\hbox{$ S  $}}.
\end{equation*}

If we extend the above definition for knotoids and mixed knotoids, and with a little abuse of notation, we may say that the Kauffman bracket skein module of $S^2$ is freely generated by the trivial knotoid and that the Kauffman bracket skein module of T is freely generated by an infinite set of generators $\left\{x^n\right\}_{n=0}^{\infty}$, where $x^n$ denotes a parallel copy of $n$ longitudes of T and $x^0$ is the trivial knotoid (see Figure~\ref{tur}).

\begin{figure}[!ht]
\begin{center}
\includegraphics[width=4.6in]{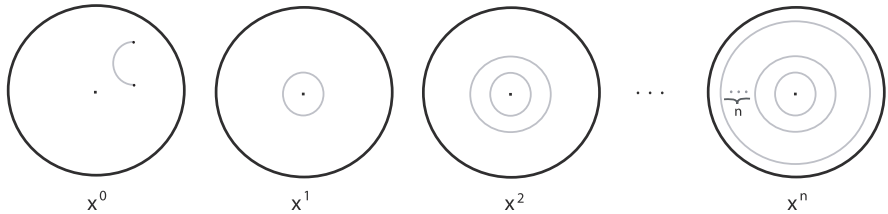}
\end{center}
\caption{The basis of KBSM(T).}
\label{tur}
\end{figure}

It is worth mentioning that in \cite{D4}, a different basis for the classical Kauffman bracket skein module of ST, elements of which are presented in Figure~\ref{diam} in terms of knotoids on T. Note that in this setting, the bold unknot represents the complementary ST in $S^3$.

\begin{figure}[H]
\begin{center}
\includegraphics[width=4.6in]{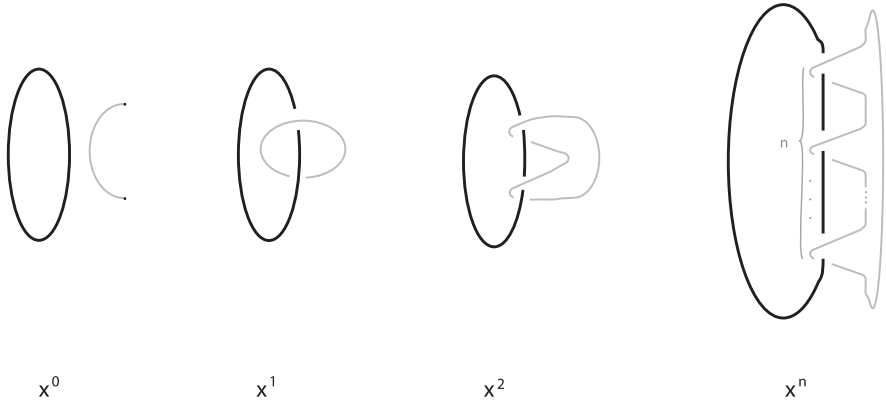}
\end{center}
\caption{A different basis of KBSM(T).}
\label{diam}
\end{figure}

\begin{remark}\rm
The computation of skein modules is a difficult task in general. The braid approach using mixed braids, mixed braid groups and appropriate knot algebras has allowed us to compute the Kauffman bracket skein module of the Solid Torus in \cite{D4} and of the handlebody of genus 2 in \cite{D5}. Moreover, for the case of HOMFLYPT skein modules, which are even more difficult to compute, the braid approach has been successfully applied so far for the case of the Solid Torus in \cite{DL2}, and significant steps toward the computation of the HOMFLYPT skein module of the lens spaces $L(p,1)$ have been done in \cite{DL3, DL4, DLP, D3}.
\end{remark}

\subsection{Mixed pseudo knotoids}\label{mpkd1}

In this subsection we generalize the notion of mixed knotoids to that of mixed pseudo knotoids. More precisely, we define a mixed pseudo knotoid to be a mixed knotoid with some crossing information missing (for an illustration see Figure~\ref{mpknoid}). 

\begin{figure}[!ht]
\begin{center}
\includegraphics[width=1.9in]{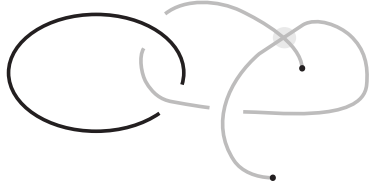}
\end{center}
\caption{A mixed pseudo knotoid.}
\label{mpknoid}
\end{figure}

Pseudo mixed knotoid isotopy is defined in a similar way as in Theorem~\ref{RthmpST}, that is, the equivalence moves involved must keep $\widehat{I}$ pointwise fixed. More precisely, in terms of pseudo diagrams we have the following result for isotopy of mixed pseudo knotoids:

\begin{thm}[{\bf The analogue of the Reidemeister theorem for mixed pseudo knotoids}]\label{RthmpST}
Two mixed pseudo knotoids are isotopic if and only if they iffer by a finite sequence of the following moves:
\smallbreak
\begin{itemize}
\item[i.] The classical Reidemeister moves and the pseudo-Reidemeister moves that only involve the moving part of the mixed pseudo knotoid and that take place away from the endpoints of the mixed pseudo knotoid (recall Figure~\ref{reid}).
\smallbreak
\item[ii.] The generalized Reidemeister moves $GR_2$ and $GR_3$, and the pseudo generalized move $PGR_3$, that involve both the fixed and the moving part of the mixed knotoid. For an illustration of the $PGR_3$ move see Figure~\ref{pgr33}.
\smallbreak
\item[iii.] The {\it generalized (or mixed) fake forbidden moves} (recall Figure~\ref{almo}).
\end{itemize}
\end{thm}

\begin{figure}[!ht]
\begin{center}
\includegraphics[width=1.6in]{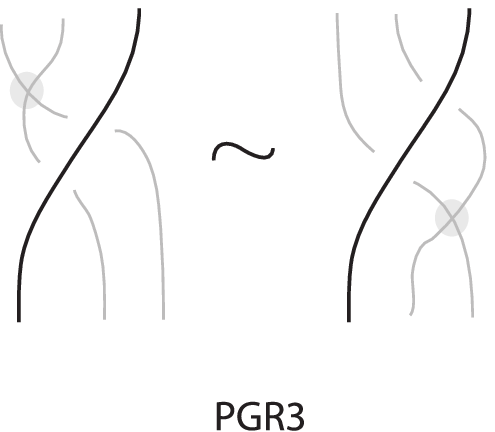}
\end{center}
\caption{The pseudo generalized Reidemeister 3 move.}
\label{pgr33}
\end{figure}

We may now extend the definition of the Kauffman bracket polynomial for mixed pseudo knotoids, in a similar way as in Definition~\ref{pkaufbst}. The main and only difference is that the orientation of a diagram in the case of mixed pseudo knotoids is needed in order to define a skein relation on pre-crossings (for more details the reader is referred to \cite{HD, D2}). More precisely, we have the following:

\begin{defn}\label{pkaufbst1}\rm
Let $L$ be an oriented mixed pseudo link in $S^2$. The {\it pseudo bracket polynomial} of $L$ is defined by means of the following relations:
\begin{figure}[H]
\begin{center}
\includegraphics[width=3.6in]{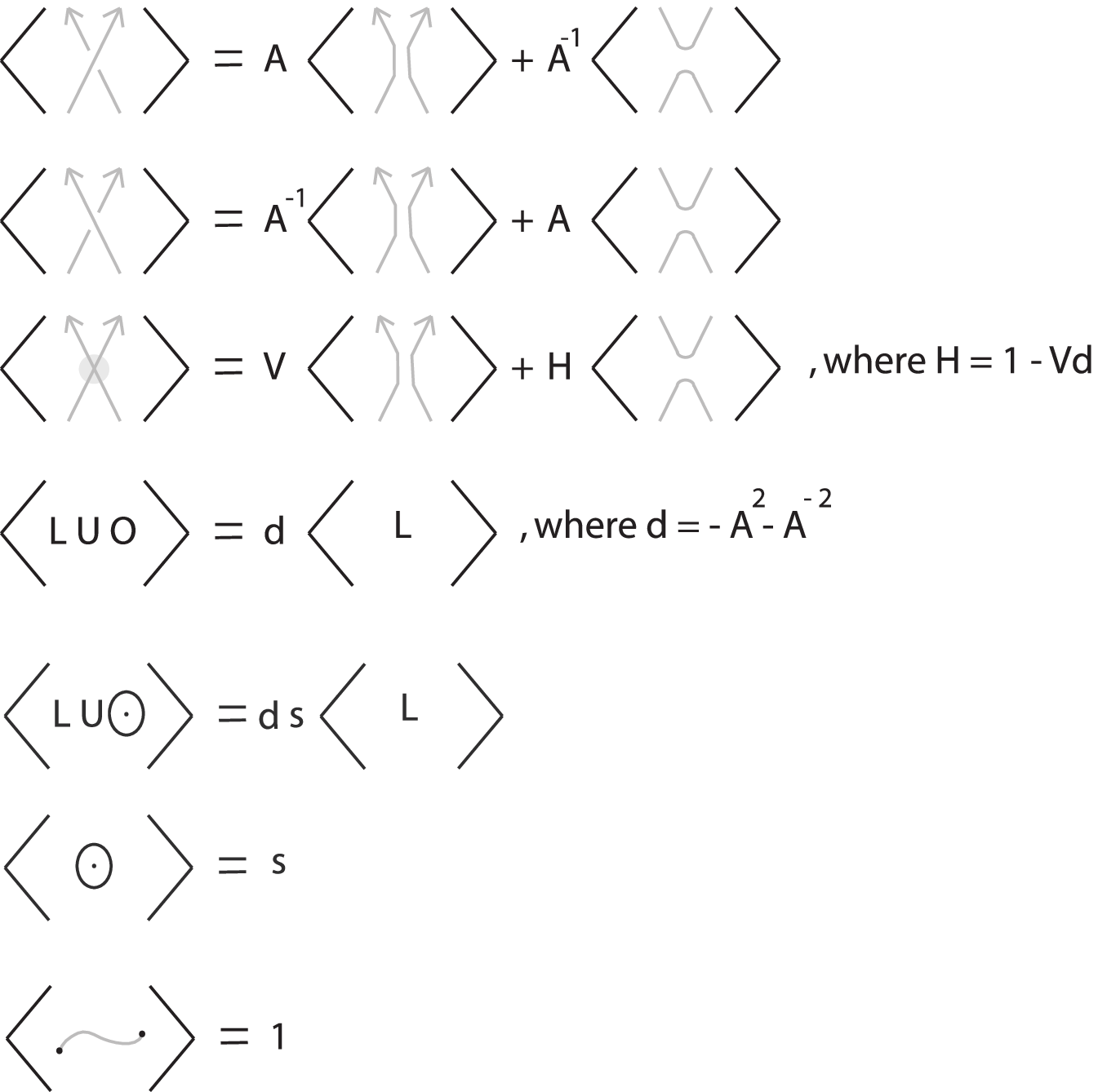}
\end{center}
\label{pkbstoi}
\end{figure}
\end{defn}

As in the case of mixed knotoids in $S^2$, we may normalize $<L>$ using the writhe of $L$ and we obtain the following:

\begin{thm}
Let $K$ be a mixed pseudo diagram of a mixed pseudo knotoid in $S^2$. The polynomial
\[
P_K(A, V, s)\ =\ (-A^3)^{w(K)}\, <K>,
\]
\noindent where $w(K)$ is the writhe of the mixed pseudo knotoid and $<K>$ the mixed pseudo bracket polynomial of $K$, is an invariant of mixed pseudo knotoids.
\end{thm}

\section{Mixed Braidoids}\label{broid11}

In this section we define {\it mixed braidoids} as a counterpart theory of mixed knotoids, just as mixed braids comprise an (algebraic) counterpart of mixed links. Assuming that the fixed strand $I$ is oriented downwards we may define a {\it mixed braidoid diagram} on $n$-strands, denoted by $I\cup B$, to be a braidoid diagram consisting of two disjoint sets of strands: one (fixed) strand forms the identity braid $I$ and it represents the complementary solid torus, and the other set of (moving) strands represents the knotoid on T and it consists of common braid strands and two free strands as in the case of classical braidoids (recall Definition~\ref{broiddefn}). For an illustration see the middle of Figure~\ref{lm}.

\smallbreak

We now translate isotopy for mixed knotoids on the level of mixed braidoids.

\smallbreak

\begin{defn}\label{broidisost}\rm
Two mixed braidoid diagrams are said to be isotopic if one can be obtained from the other by a finite sequence of the following moves, that we call {\it mixed braidoid isotopy} moves:
\smallbreak
\begin{itemize}
\item[$\bullet$] {\it Moving part:} We allow all moves of Definition~\ref{broidiso} on the moving strands of a mixed braidoid.
\smallbreak
\item[$\bullet$] {\it Fixed \& Moving part:} We allow the endpoints of the knotoid to be pushed or be pulled over or under the fixed strand of a mixed braidoid. This can be realized as a special case of vertical or swing moves and we call such moves {\it generalized swing moves}. The generalized swing moves translate the generalized fake forbidden moves in terms of mixed braidoids and they are illustrated in the bottom of Figure~\ref{almo1}. Moreover, the generalized Reidemeister moves are also allowed and they form a special case of $\Delta$-moves (for an illustration see the top part of Figure~\ref{almo1}).
\begin{figure}[ht]
\begin{center}
\includegraphics[width=5in]{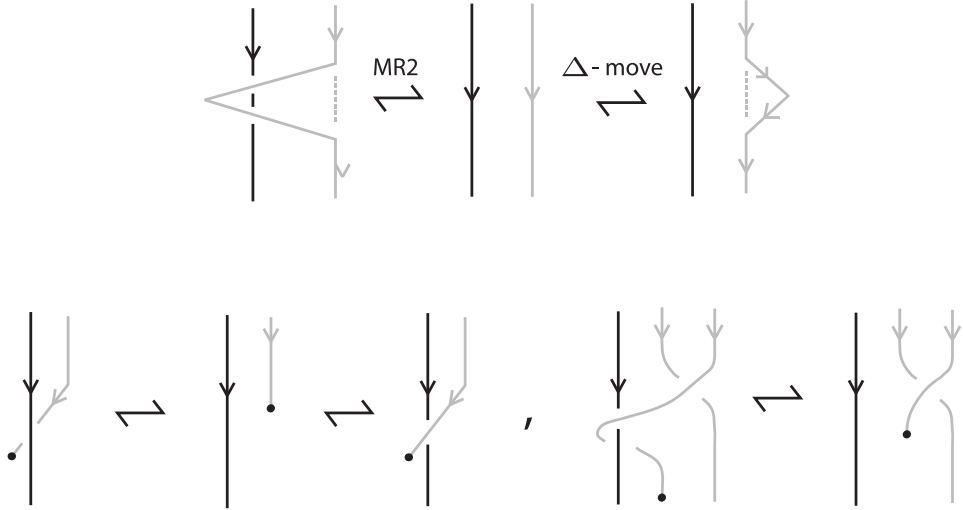}
\end{center}
\caption{The generalized Reidemeister move MR2 and the generalized swing moves.}
\label{almo1}
\end{figure}
\end{itemize}
\smallbreak
An isotopy class of mixed braidoid diagrams is called a {\it mixed braidoid}. Moreover, a {\it labeled mixed braidoid diagram} is a mixed braidoid diagram with a label over or under assigned to each pair of corresponding ends, not including the fixed strand.
\end{defn}

The reason why the fixed strand remains unlabeled will become clear in the next subsection.

\subsection{The closure operation}

We now define a closure operation on mixed braidoids, similar to that of mixed braids in handlebodies. The main difference is that the closure operation on a mixed braidoid should take into consideration the forbidden moves that may occur after this operation is applied on a mixed braidoid. We have the following:

\begin{defn}\label{clbd}\rm
The {\it closure} $\mathcal{C}(I \cup B)$ of a mixed braidoid $I\cup B$ is an operation that results in a (oriented) mixed knotoid in $S^2$ and is defined as follows: each pair of corresponding ends of the moving part of the mixed braidoid are joined by embedded arcs that run along the right hand-side of the vertical line of the corresponding mixed braidoid ends, either over or under the rest of the mixed braidoid and according to the label attached to these endpoints. It is crucial to note that in order to avoid forbidden moves on the closure of the mixed braidoids, the embedded arcs that appear after the closure operation is applied, lie closer to the mixed braidoid ends they connect compared to the distance of the leg or head from the vertical line passing through the pair of the mixed braidoid ends they connect. Note also that due to the generalized forbidden moves, the closure of the fixed strand may be realized as in the case of classical mixed links, that is, by an arc at infinity that identifies the two horizontal arcs that contain the endpoints of $I$. For an illustration see Figure~\ref{clmbr}. For this reason the fixed strand remains unlabeled. Finally, note that the endpoints of the mixed braidoid do not participate in the closure operation and they form the endpoints of the resulting mixed knotoid.
\end{defn}

\begin{figure}[ht]
\begin{center}
\includegraphics[width=5in]{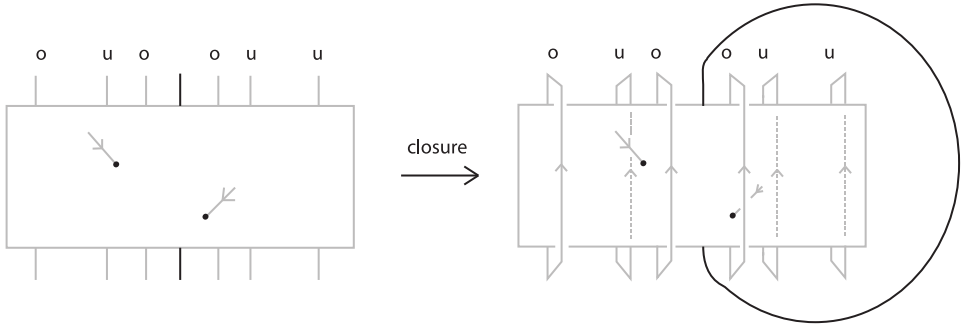}
\end{center}
\caption{The closure of a mixed braidoid.}
\label{clmbr}
\end{figure}

\begin{remark}\rm
Note that we may define the closure of $I$ as in the case of the moving strands of the mixed braidoid, where the label used is irrelevant. This follows from the generalized forbidden moves that allow the closing arc of the fixed strand to slide freely to the side of the mixed braidoid over or under the rest of the mixed braidoid as illustrated in Figure~\ref{clfstr1}.
\end{remark}

\begin{figure}[ht]
\begin{center}
\includegraphics[width=5in]{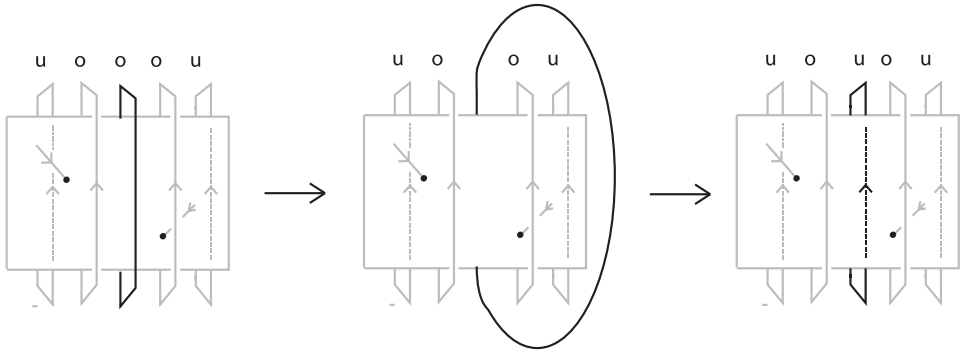}
\end{center}
\caption{The unlabelled closure of the fixed strand.}
\label{clfstr1}
\end{figure}

Note now that isotopy moves on (labeled) mixed braidoids can be translated to isotopy moves on the resulting mixed knotoid diagrams, and thus, we have the following result:

\begin{prop}
The closure operation of Definition~\ref{clbd} is a well defined map:
\[
\mathcal{C}\, :\, I\cup B\ \rightarrow\ \widehat{I} \cup L, 
\]
\noindent where $I\cup B$ denotes the set of (labeled) mixed braidoids and $\widehat{I} \cup L$ denotes the set mixed knotoid diagrams.
\end{prop}

We conclude this subsection with an important theorem which comprise the bridge between mixed knotoids and mixed braidoids, namely, the analogue of the Alexander theorem for mixed knotoids.

\begin{thm}[{\bf The analogue of the Alexander theorem for mixed knotoids}]\label{alexbrst}
Any mixed knotoid diagram is isotopic to the closure of a (labeled) mixed braidoid diagram.
\end{thm}

\begin{proof}
The proof is similar to the braidoiding algorithm presented in \cite{GL1}. This braidoiding algorithm may be applied in the case of mixed knotoids since it will not affect their fixed part $\widehat{I}$. This is because the up-arcs on a mixed knotoid are arcs of its moving part. The result follows.
\end{proof}

\begin{remark}\rm
It is worth mentioning that the braiding algorithm of \cite{La2} may also be applied for the case of mixed braidoids. This is because of the definition of the closure of a mixed braidoid. The same braiding algorithm has been applied in \cite{D2} for the case of pseudo links and singular links in the Solid Torus.
\end{remark}

\subsection{An analogue of the Markov theorem for mixed knotoids}

In this subsection we formulate and prove a geometric analogue of the Markov theorem for mixed knotoids. For this we need to introduce first the $L$-moves on mixed braidoids. $L$-moves make up an important tool for formulating braid equivalences in any topological setting and they prove to be particularly useful in settings where the sets of braid analogues do not have a ``nice'' algebraic structure. The innterested reader is referred to \cite{LR1, La1}, where $L$-moves and braid equivalence theorems are presented for different knot theories.

\begin{defn}\label{lmdefn}\rm
An {\it $L$-move} on a labeled mixed braidoid $I\cup B$, is an $L$-move on a moving strand of the mixed braidoid, as illustrated in Figure~\ref{lm} by omitting the labels on the moving strands of the mixed braidoid.
\end{defn} 

\begin{figure}[ht]
\begin{center}
\includegraphics[width=4.9in]{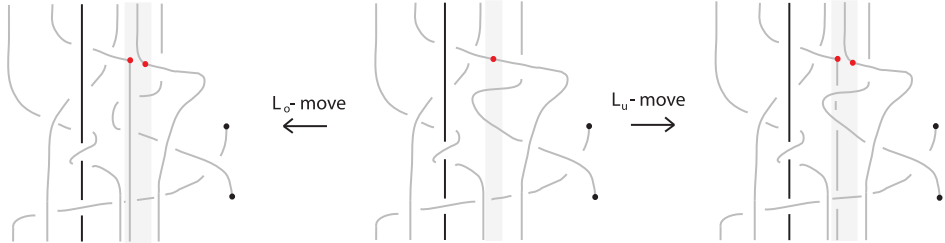}
\end{center}
\caption{$L$-moves on mixed braidoids.}
\label{lm}
\end{figure}

The $L$-moves allow us to formulate and prove a geometric analogue of the Markov theorem for mixed braidoids. More precisely, we have the following theorem:

\begin{thm}[{\bf Geometric analogue of the Markov theorem for mixed braidoids}]\label{brisomarkst}
Two (oriented) mixed knotoids are isotopic if and only if any two corresponding mixed braoidoids of theirs differ by a finite sequence of $L$-moves and isotopies of mixed braidoids.
\end{thm}

\begin{proof}
In \cite{LR1} the authors prove a relative version of the analogue of the Markov theorem for knots and links in $S^3$ (Theorem~4.7 \cite{LR1}). More precisely, they prove that two oriented link diagrams that contain a common braided portion $B$ are isotopic if and only if corresponding braids of theirs, that are obtained through the braiding algorithm of \cite{LR1}, differ by $L$-moves that do not affect the common braided portion $B$. This result is due to the 1-move Markov theorem for knots and links in $S^3$ using the $L$-moves and the braiding algorithm applied. Since now the braidoiding algorithm of \cite{GL1} is similar to that in \cite{LR1} and since the braidoiding moves applied on a knotoid in the braidoiding algorithm are similar to the $L$-moves, it follows that the same result holds for knotoids. That is:

\begin{prop}[{\bf Relative version of the Markov theorem for braidoids}]
Let $L_1, L_2$ be knotoids in $S^2$, both containing a common braided portion $B$, and let $L_1$ be isotopic to $L_2$, where the isotopy involved finishes with a homeomorphism fixed on $B$. Then, if $B_1, B_2$ are two braidoids obtained from the braidoiding algorithm of \cite{GL1} applied on $L_1, L_2$ respectively, both containing the braid $B$, then $B_1$ is $L$-equivalent to $B_2$ by moves that do not affect the braid $B$.
\end{prop}

Recall now that all mixed braidoids have a common braided part, $I$, which remains fixed throughout an isotopy of two mixed knotoid diagrams (recall Theorem~\ref{RthmST}). The result follows.
\end{proof}

\subsection{Mixed pseudo braidoids}

In this subsection we introduce the theory of {\it mixed pseudo braidoids} as the counterpart theory of mixed pseudo knotoids. A mixed pseudo braidoid is defined to be a (labeled) mixed braidoid with some crossing information missing. The closure operation on a mixed pseudo braidoid is defined in the same way as in the case of mixed braidoids by ignoring the pre-crossings.

\smallbreak

The braidoiding algorithm of \cite{GL1}, that was also applied in the case of mixed braidoids, may also be applied on pseudo mixed knotoids. In order to use this algorithm on mixed pseudo knotoids we need to deal with pre-crossings in the diagram which contain at least one up-arc first. For this we apply the idea used in \cite{KL} for the case of virtual knots. Namely, before we apply the braidoiding algorithm we have to isotope the mixed pseudo knotoid in such a way that the pre-crossings will only contain down-arcs, so that the braidoiding algorithm will not affect them. This is achieved by rotating all pre-crossings that contain at least one up-arc, so that the two arcs are now directed downward, as illustrated in Figure~\ref{tw}. Then we may apply the braidoiding algorithm of \cite{LR1, GL1} for the mixed knotoid by ignoring the pre-crossings.

\begin{figure}[ht]
\begin{center}
\includegraphics[width=4.7in]{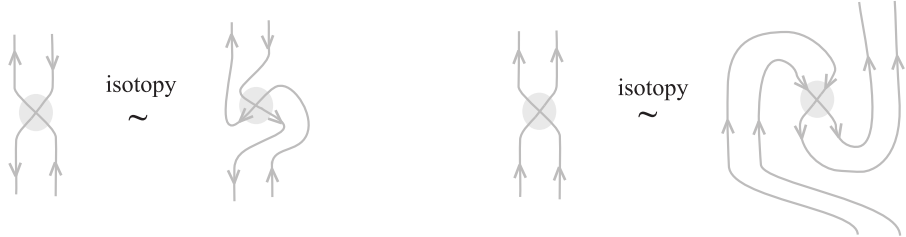}
\end{center}
\caption{Rotating pre-crossings.}
\label{tw}
\end{figure}

The discussion above provides a proof for the following theorem:
 
\begin{thm}[{\bf The analogue of the Alexander theorem for mixed pseudo knotoids}]\label{newprpkalex}
Every oriented mixed pseudo knotoid is isotopic to the closure of a mixed pseudo braidoid.
\end{thm}

\section{Conclusions \& Further Research}

In this paper we introduce and study the theory of mixed knotoids and mixed pseudo knotoids on $S^2$. In particular, we study how isotopy is translated in terms of moves on mixed (pseudo) knotoid diagrams and we also extend the Kauffman bracket polynomial for mixed knotoids and for mixed pseudo knotoids. We then introduce the notions of mixed braidoids and mixed pseudo braidoids, and using a braidoiding algorithm, we relate mixed braidoids to mixed knotoids and mixed pseudo braidoids to mixed pseudo knotoids. Finally, we formulate and prove a geometric analogue of the Markov theorem using the $L$-moves for the case of mixed braidoids. Note that a similar construction would lead to the theory of {\it knotoids on handlebodies}, which is the subject of a sequel paper.

\smallbreak

Moreover, in \cite{D1} the theory of {\it tied pseudo links} is introduced and in \cite{D} the theory of {\it tied links} is extended to other 3-manifolds. Tied links are classical links equipped with {\it ties}, that is, non-embedded arcs joining some components of the link. Our intention is to study the theories of {\it tied knotoids} and {\it tied pseudo knotoids} in $S^2$.


\begin{thebibliography}{99}

\bibitem{D} {\sc I. Diamantis}, Tied links in various topological settings, arXiv:2010.00374 (2020).

\bibitem{D1} {\sc I. Diamantis}, Tied pseudo links \& pseudo knotoids, {\it Mediterranean Journal of Mathematics}, to appear. arXiv:2010.06162v2 (2020).

\bibitem{D2} {\sc I. Diamantis}, Pseudo links and singular links in the Solid Torus, arXiv:2101.03538v1.

\bibitem{D3} {\sc I. Diamantis}, HOMFLYPT skein sub-modules of the lens spaces $L(p,1)$ via braids, {\em Topology and its Applications}, 2020,
107500, ISSN 0166-8641, https://doi.org/10.1016/j.topol.2020.107500.

\bibitem{D4} {\sc I. Diamantis} (2019), An Alternative Basis for the Kauffman Bracket Skein Module of the Torus via Braids. In: Adams C. et al. (eds) Knots, Low-Dimensional Topology and Applications. KNOTS16 2016. Springer Proceedings in Mathematics \& Statistics, vol 284. Springer, Cham. https://doi.org/10.1007/978-3-030-16031-9\textunderscore 16.

\bibitem{D5} {\sc I. Diamantis}, The Kauffman bracket skein module of the handlebody of genus 2 via braids, {\em J. Knot Theory and Ramifications}, {\bf 28}, No. 13, 1940020 (2019). doi:$10.1142/S0218216519400200$.


\bibitem[DL1]{DL1} {\sc I. Diamantis, S. Lambropoulou}, Braid equivalences in 3-manifolds 
with rational surgery description, {\em Topology and its Applications}, {\bf 194} (2015), 269-295. doi:$10.1016/j.topol.2015.08.009$.

\bibitem{DL2} {\sc I. Diamantis, S. Lambropoulou}, A new basis for the HOMFLYPT skein module of the Solid Torus, {\em J. Pure Appl. Algebra} {\bf 220} Vol. 2 (2016), 577--605. doi:$10.1016/j.jpaa.2015.06.014$.

\bibitem{DL3} {\sc I. Diamantis, S. Lambropoulou}, The braid approach to the HOMFLYPT skein module of the lens spaces $L(p, 1)$, Springer Proceedings in Mathematics and Statistics (PROMS),{\em Algebraic Modeling of Topological and Computational Structures and Application}, (2017). doi:$10.1007/978-3-319-68103-0_7$.

\bibitem{DL4} {\sc I. Diamantis, S. Lambropoulou}, An important step for the computation of the HOMFLYPT skein module of the lens spaces $L(p,1)$ via braids, {\em J. Knot Theory and Ramifications}, {\bf 28}, No. 11, 1940007 (2019). doi:$10.1142/S0218216519400078$.

\bibitem{DGBS} {\sc J. Dorier, D. Goundaroulis, F. Benedetti \& A. Stasiak}, Knotoid: a tool to study the entanglement of open protein chains using the concept of knotoids, \emph{Bioinformatics}, {\bf 34} (2018) No. 19, pp. 3402--3404.

\bibitem{DLP} {\sc I. Diamantis, S. Lambropoulou, J. H. Przytycki}, Topological steps on the HOMFLYPT skein module of the lens spaces $L(p,1)$ via braids, {\it J. Knot Theory and Ramifications}, {\bf 25}, No. 14, (2016). doi:$10.1142/S021821651650084X$.

\bibitem{GGLDSK} {\sc D. Goundaroulis, N. G\H{u}g\H{u}mc\H{u}, S. Lambropoulou, J. Dorier, A. Stasiak \& L. Kauffman}, Topological models for open-knotted protein chains using the concepts of knotoids and bonded knotoids, \emph{Polymers}, {\bf 9} (2017) No. 9, 444. 

\bibitem{GL1} {\sc N. G\H{u}g\H{u}mc\H{u} \& S. Lambropoulou}, Knotoids, braidoids and applications, {\it Symmetry} {\bf 9} (2017), No.12, 315.

\bibitem{GL2} {\sc N. G\H{u}g\H{u}mc\H{u} \& S. Lambropoulou}, Braidoids, {\it Israel J. Math.}; to appear, arXiv:1908.06053v2.

\bibitem{H} {\sc R. Hanaki}, Pseudo diagrams of links, links and spatial graphs, {\it Osaka J. Math.}, {\bf 47}, (2010) 863--883.

\bibitem{HD} {\sc H. A. Dye}, Pseudo knots and an obstruction to cosmetic crossings, {\it J. Knot Theory and Ramifications}, {\bf 26}, No. 04, (2017) 1750022

\bibitem{HJMR} {\sc A. Henrich, R. Hoberg, S. Jablan, L. Johnson, E. Minten, L. Radovic}, The theory of pseudoknots, {\it J. Knot Theory and Ramifications}, {\bf 22}, No. 07, (2013) 1350032.

\bibitem{HL}{\sc R. H\"{a}ring-Oldenburg \& S. Lambropoulou}, Knot theory in handlebodies, {\it J. Knot Theory and its Ramifications} {\bf 6}, No. 6, (2002) 921--943.
		
\bibitem{KL} {\sc L.H. Kauffman \& S. Lambropoulou}, Virtual braids and the $L$-move, \emph{Journal of Knot Theory and Its Ramifications} {\bf 15}, (2006) No. 6, 773--811.

\bibitem{La1} {\sc S. Lambropoulou}, $L$-Moves and Markov theorems, \emph{Journal of Knot Theory and Its Ramifications} {\bf 16}, Ni.10, 
(2007) 1459--1468.							

\bibitem{La2} {\sc S. Lambropoulou}, Solid torus links and Hecke algebras of B-type, {\it Quantum Topology}; D.N. Yetter Ed.; World Scientific Press, (1994), 225-245.

\bibitem{LK} {\sc L. H. Kauffman}, New invariants in the theory of knots, \emph{Amer. Math. Monthly} {\bf 95}, (3) (1988) 195--242.

\bibitem{LR1} {\sc S. Lambropoulou, C.P. Rourke} (2006), Markov's theorem in $3$-manifolds, \emph{Topology and its Applications} {\bf 78},
(1997) 95--122.

\bibitem{P} {\sc J.~Przytycki}, Skein modules of 3-manifolds, {\it Bull. Pol. Acad. Sci.: Math.}, {\bf 39, 1-2} (1991), 91--100.

\bibitem{T} {\sc V. Turaev}, Knotoids, \emph{Osaka J. Math.}, {\bf 49} (2012), pp. 195--223.

\bibitem{Tu} {\sc V.G.~Turaev}, The Conway and Kauffman modules of the Solid Torus,  {\it Zap. Nauchn. Sem. Lomi} {\bf 167} (1988), 79--89. English translation: {\it J. Soviet Math.} (1990), 2799-2805.

\end{thebibliography}
\end{document}